\documentclass[12pt,reqno]{amsart}
 \usepackage[dvips]{epsfig}
 \usepackage{comment}
 \usepackage{enumerate}
 \usepackage{amsgen, amstext,amsbsy,amsopn, amsthm, amsfonts,amssymb,amscd,amsmat
 h,euscript,enumerate,url,verbatim,calc,xypic}
 
 \usepackage{latexsym}
 \usepackage{graphics}
 \usepackage{color}

\newcommand{\Pic}{{\rm Pic\,}}

\newcommand{\cc}{{\mathfrak c}} 
\newcommand{\im}{{\rm im\,}}
\newcommand{\proset}{\,\mathrel{\lower 4pt\hbox{$\scriptscriptstyle/$}
\mkern -14mu\subseteq }\,} %for proper subset
 
 \newtheorem{theorem}{Theorem}[section]
  \newtheorem{corollary}[theorem]{Corollary}
 \newtheorem{lemma}[theorem]{Lemma}
 \newtheorem{proposition}[theorem]{Proposition}

\newtheorem{remark}[theorem]{Remark}

 \newtheorem*{thm1}{Theorem{\bf $\,$ \ref{2.10}}}
\newtheorem*{thm2}{Theorem{\bf $\,$ \ref{LP1}}}
\newtheorem*{thm3}{Theorem{\bf $\,$ \ref{2.11}}}

\usepackage{amsmath}

\begin{document}
 \title{Subintegrality, Invertible Modules and Laurent Polynomial Extensions}
 \author{Vivek Sadhu}
 
 \address{Department of Mathematics, Indian Institute of Technology 
Bombay, Powai, Mumbai 400076, India} \email{viveksadhu@math.iitb.ac.in}

 \maketitle
 \begin{abstract}
 Let $A\subseteq B$ be a commutative ring extension. Let $\mathcal I(A, B)$ be the multiplicative group of invertible $A$-submodules of $B$. In this article, we extend a result of Sadhu and Singh by finding a necessary and sufficient condition on an integral birational extension $A\subseteq B$ of integral domains with $\dim A\leq 1$, so that the natural  map $\mathcal I(A,B) \rightarrow  \mathcal I (A [X, X^{-1}],B [X, X^{-1}])$ is an isomorphism. In the same situation, we show that if $\dim A\geq 2$ then the condition is necessary but not sufficient. We also discuss some properties of the cokernel of the natural map $\mathcal I(A,B) \rightarrow  \mathcal I (A [X, X^{-1}],B [X, X^{-1}])$ in the general case.
  
 \end{abstract}
 
 
 
 \section*{Introduction}
In \cite{rs}, Roberts and Singh have introduced the group $\mathcal I(A, B)$ to generalize a result of Dayton. The relation between the group $\mathcal I(A, B)$ and subintegral extensions has been investigated by Reid, Roberts and Singh in a series of papers. Recently in \cite{ss}, Sadhu and Singh have proved that $A$ is subintegrally closed in $B$ if and only if the canonical map $\mathcal I(A,B) \rightarrow  \mathcal I (A [X],B [X])$ is an isomorphism. It is easy to see that the map is injective and that $\mathcal I(A[X], B[X])= \mathcal I(A, B)\oplus N\mathcal I(A, B)$, where $N\mathcal I(A, B)$ denotes the kernel of the map $\mathcal I(A[X], B[X]) \stackrel{X\mapsto 0}\rightarrow \mathcal I(A, B)$. So the result of \cite{ss}, just mentioned,  amounts to saying that  $N\mathcal I(A, B)= 0$ if and only if $A$ is subintegrally closed in $B$.

The primary goal of this paper is to extend the result of Sadhu and Singh in \cite{ss} just mentioned above by finding a necessary and sufficient condition on $A\subseteq B$, so that the natural  map $\mathcal I(A,B) \rightarrow  \mathcal I (A [X, X^{-1}],B [X, X^{-1}])$ is an isomorphism. It is easy to see that the map  $\mathcal I(A, B)\rightarrow \mathcal I(A[X, X^{-1}],B[X, X^{-1}])$ is always injective (see Lemma \ref{2.2}). Thus the problem reduces to the investigation of conditions for the cokernel of the above map to be zero. This cokernel will be denoted by $M\mathcal I(A, B)$. The secondary goal will be to investigate properties of the  cokernel $M\mathcal I(A, B)$ in the general case.

In Section 1, we mainly give basic definitions and  notations.

In Section 2, we discuss conditions on $A\subseteq B$ under which the  map $\mathcal I(A, B)\rightarrow \mathcal I(A[X, X^{-1}],B[X, X^{-1}])$ is an isomorphism. We are able to prove some results in the situation when $A\subseteq B$ is an integral birational extension of domains. First, if $\dim A\leq 1$ then by using a result of Onoda-Yoshida (\cite{Oy}, Theorem 1.13), we prove the following 
\begin{thm3}
Let $A\subseteq B$ be an integral, birational extension of domains with $\dim A\leq 1$. Then $M\mathcal I(A, B)=0$ if and only if $A$ is subintegrally closed in $B$ and $A\subseteq B$ is anodal.
\end{thm3}
For higher dimension, we show that the above conditions are necessary but not sufficient. More precisely, we prove the following
\begin{thm1}
  Let $A\subseteq B$ be an integral, birational extension of domains. Suppose $M\mathcal I(A, B)=0$. Then  $A$ is subintegrally closed in $B$ and $A\subseteq B$ is anodal.
\end{thm1}
That the conditions are not sufficient is shown by an example of C. Weibel (see Remark \ref{example}).
We note that for any ring extension $A\subseteq B$, the condition $M\mathcal I(A, B)=0$ implies easily that $A$ is subintegrally closed in $B$ (see Lemma \ref{2.8}(4)).
 
 In Section 3, we examine the cokernel $M\mathcal I(A, B)$ in the general case. In order to do this, we first discuss the surjectivity of the natural  map $\varphi(A, C, B) : \mathcal I(A, B)\rightarrow \mathcal I(C, B)$ given by $\varphi(A, C, B)(I)= IC$ for any ring extensions $A\subseteq C \subseteq B$. We show that the map $\varphi(A, C, B)$ is surjective in two cases: (1) $C$ is subintegral over $A$, (2) $A\subseteq B$ is an integral extension with $A$ Hensel local (see Propositions \ref{3.5} and \ref{3.6}). We show further that if $C$ is subintegral over  $A$, then the sequence 
 $$ 1\rightarrow M\mathcal I(A,C) \rightarrow M\mathcal I(A, B) \rightarrow M\mathcal I(C, B) \rightarrow 1$$
 is exact (see Proposition \ref{3.7}). Finally we prove the following 
\begin{thm2}
  Let $A\subseteq B$ be a ring extension with $A$ Hensel local and $B$ seminormal. Then $M\mathcal I(A, B)\cong M\mathcal I(A,^{^+}\!\!\!A)\oplus M\mathcal I(^{^+}\!\!\!A, B)$, where $^{^+}\!\!\!A$ is the subintegral closure of $A$ in $B$. 
\end{thm2}
In this section we also observe that if $A$ is subintegrally closed in $B$ with $B$ a seminormal domain and $A$ Hensel local then $M\mathcal I(A, B)=0$ (see Proposition \ref{4.8}(4)).

 \section{Basic definitions and Notations}
 All the rings we consider are commutative with $1$, and all ring homomorphisms are unitary. Let $X$, $T$ be  indeterminates.\
  
 An {\bf elementary subintegral }extension is an extension of the form $A\subseteq B$ with $B= A[b]$ for some $b\in B$ such that $b^{2},b^{3}\in A$. An extension $A\subseteq B$ is {\bf subintegral} if it is a filtered union of elementary subintegral extensions; that is, for each $b\in B$ there is a finite sequence $A = C_{0}\subseteq C_{1}\subseteq \cdots \subseteq C_{r}\subseteq B$ of ring extensions such that $b\in C_{r}$ and $C_{i-1}\subseteq C_{i}$ is elementary subintegral for each $i$, $1\leq i \leq r $. We say that $A$ is {\bf subintegrally closed} in $B$ if whenever $b\in B$ and $b^{2}, b^{3}\in B$ then $b\in A$. The ring   $A$ is {\bf seminormal} if the following condition holds: $b,c\in A$ and $b^3=c^2$ imply that there exists $a\in A$ with $b=a^2$ and $c=a^3$.  A seminormal ring is necessarily reduced and is subintegrally  closed  in every reduced overring.  It is easily seen that if $A$ is subintegrally closed in $B$ with $B$ seminormal then $A$ is seminormal. For details see \cite{swan, mv}.\

  For a ring $A$ we denote by:
  
  $U(A)$:    The groups of units of $A$.
  
  $H^{0}(A)= H^{0}(\rm{Spec}(A), \mathbb Z)$: The group of continuous maps from  $\rm{Spec}(A)$ to $ \mathbb Z$.
  
  $\rm{\Pic}A$: The Picard group of $A$.
   
  $KU(A)$: Cokernel of the natural map $U(A)\rightarrow U(A[X])$.
  
  $MU(A)$:  Cokernel of the natural map $U(A)\rightarrow U(A[X, X^{-1}])$.
  
  $NU(A)$: Kernel of the map $U(A[X]) \rightarrow U(A)$.
  
  $\rm{K\Pic}A$:  Cokernel of the natural map $\Pic A\rightarrow \Pic A[X]$.
  
  $\rm{M\Pic} A$: Cokernel of the natural map $ \Pic A\rightarrow \Pic A[X, X^{-1}] $.
   
  $\rm{N\Pic} A$:  Kernel of the map $\Pic A[X]\rightarrow \Pic A$.
  
  $\rm{L\Pic }A$:  Cokernel of the map $\Pic A[X]\times \Pic A[X^{-1}]\stackrel{add}\rightarrow \Pic A[X, X^{-1}]$.

  \hspace{2cm}
   
  Let $A\subseteq B$ be a ring extension. Then we denote  by 
  
   $ \mathcal I(A,B) $:  The group of invertible $A$-submodules of $B$.
   
    It is easily seen that $\mathcal I$ is a functor from extensions of rings to abelian groups. Some properties of $\mathcal I(A, B)$ can be found in [4, Section 2].
    
     $K\mathcal I(A,B)$: Cokernel of the natural map $\mathcal I(A,B) \rightarrow  \mathcal I (A [X],B [X])$.
     
   $M \mathcal I(A,B) $: Cokernel of the natural map $\mathcal I(A,B) \rightarrow  \mathcal I (A [X, X^{-1}],B [X, X^{-1}])$.
   
   $N \mathcal I(A, B)$: Kernel of the map $\mathcal I(A[X], B[X]) \rightarrow \mathcal I(A, B)$ (Here the map is  induced by the B-algebra  homomorphism $B [X]\rightarrow  B$ given by $X\mapsto 0$).

 Recall from[4, Section 2] that for any commutative ring extension $A\subseteq B$, we have the exact sequence 
 $$ 
 1\rightarrow  U(A) \rightarrow  U(B) \rightarrow  \mathcal  I(A,B) \rightarrow  \Pic A \rightarrow  \Pic B. 
$$
Applying $M$, $K$ we obtain the chain complexes:
$$ 
 (1.0)\hspace{2cm}     1\rightarrow  MU(A) \rightarrow  MU(B) \rightarrow  M \mathcal  I(A,B) \stackrel{\eta}\rightarrow  M \Pic A \stackrel{\varphi}\rightarrow  M \Pic B
$$
and
$$ 
 (1.1)\hspace{2cm}     1\rightarrow  KU(A) \rightarrow  KU(B) \rightarrow  K \mathcal  I(A,B) \stackrel{\alpha}\rightarrow  K \Pic A \stackrel{\beta}\rightarrow  K \Pic B.
$$

\section{The map $\mathcal I(A, B)\rightarrow \mathcal I(A[X, X^{-1}], B[X, X^{-1}])$}
In this section we examine some conditions on $A\subseteq B$ under which the natural map \\ $\mathcal I(A, B)\rightarrow \mathcal I(A[X, X^{-1}], B[X, X^{-1}])$ is an isomorphism. For this  we consider the notions of quasinormal and anodal extensions (or $u$-closed).

Let $A\subseteq B$ be a ring extension. We say that $A$ is $\bf{quasinormal}$ in $B$ if the natural map $M\Pic A\rightarrow M\Pic B$ is injective. For properties of quasinormal extensions see \cite{Oy}.

An inclusion $A\subseteq B$ of rings is called $\bf{anodal}$ or $\bf{an\, \,anodal\,\, extension}$, if every $b\in B$ such that $(b^{2}- b)\in A$ and $(b^{3}- b^{2})\in A$ belongs to $A$. This notion was first introduced by Asanuma and Onoda-Yoshida in \cite{Oy}, and they called this notion  \textquoteleft $u$-closed\textquoteright. Some related details can be found in \cite{TA, Oy, wei}.

We first show in Proposition \ref{4.1} below that a subintegral extension is always an anodal extension, which is perhaps a result of independent interest.
\begin{lemma}\label{2.7}
 Let $A\subseteq C \subseteq B$ be extensions of rings. Then the following statements hold:
 \begin{enumerate}[(1)]
  \item If $A$ is anodal in $B$, then so is $A$ in $C$.
  \item If $A$ is anodal in $C$ and $C$ is anodal in $B$, then so is $A$ in $B$.
 \end{enumerate}

\end{lemma}
\begin{proof}
 Clear from the definition.
\end{proof}

\begin{proposition}\label{4.1}
 Let $A\subseteq B$ be a ring extension. If $A\subseteq B$ is subintegral, then it is anodal.
 
\end{proposition}
\begin{proof}
 Assume first that $A\subseteq B$ is an elementary subintegral extension, i.e., $A\subseteq B= A[b]$ such that $b^{2}, b^{3}\in A$. Let $f\in B$ such that $f^{2}- f, f^{3}- f^{2}\in A$. We have to show that $f\in A$. Clearly $f$ is of the form $a + \lambda b$ where $a, \lambda \in A$. So it is enough to show that $\lambda b\in A$. Since $\lambda b(2a- 1), \lambda b (3a^{2} -1) \in A$,  $\lambda b= \lambda b. 1=\lambda b[(6a +3) (2a- 1) - 4(3a^{2}- 1)]\in A$. Hence $f\in A$.\
 
 In the general case, for $f\in B$ there exists a finite sequence $A= C_{0}\subseteq C_{1}\subseteq......\subseteq C_{r}\subseteq B$ of extensions such that $C_{i}\subseteq C_{i+ 1}$ is an elementary subintegral extension for each $i, 0\leq i\leq r- 1$  and $f\in C_{r}$. So by the above argument $C_{i}\subseteq C_{i +1}$ is anodal for each $i$. Now the result follows from  Lemma \ref{2.7}(2).
\end{proof}

 The following result is due to Sadhu and Singh (\cite{ss}, Theorem 1.5) which we use frequently throughout this paper:
 
 \begin{lemma}\label{2.1}
 Let $A\subseteq  B$ be a ring extension. Then $A$ is subintegrally closed in $B$ if and only if the canonical map $\mathcal I(A,B) \rightarrow  \mathcal I (A [X],B [X])$ is an isomorphism.
\qed
 \end{lemma}

 \noindent One can restate the above result in the following way: $A$ is subintegrally closed in $B$ $\Leftrightarrow K\mathcal I(A, B)=0 \Leftrightarrow   N\mathcal I(A, B)= 0$.

 The following result is due to Weibel (\cite{wei}, Theorem 5.2).
 
 \begin{lemma}\label{we}
  There is a natural decomposition $$ \Pic A[X, X^{-1}]\cong \Pic A \oplus N\Pic A \oplus N\Pic A \oplus L\Pic A $$
for any commutative ring $A$.
\qed
 \end{lemma}

 \begin{remark}\label{we1}\rm{
 By Swan Theorem \cite{swan}, $N\Pic A =0$ if and only if $A_{red}$ is seminormal. So for a seminormal ring $A$, $L\Pic A\cong M\Pic A$. }
 \end{remark}

The next result is given in (\cite{wei 1}, Exercise 3.17, Page 30).
 
 \begin{lemma}\label{LU}
  There is a natural decomposition 
  $$U(A[X, X^{-1}])\cong U(A)\oplus NU(A)\oplus NU(A)\oplus H^{0}(A)$$ for any commutative ring $A$. 
  \qed
 \end{lemma}

 \begin{remark}\label{LU1}\rm{
   It follows that for a reduced ring $A$, $H^{0}(A)\cong MU(A)$.}
 \end{remark}

 \begin{lemma}\label{2.2}
  The natural map $\phi: \mathcal I(A,B) \rightarrow  \mathcal I (A [X, X^{-1}],B [X, X^{-1}])$, given by  $I\rightarrow IA[X, X^{-1}] $, is injective. Thus, $\phi $ is an isomorphism if and only if $M\mathcal I(A, B)=0$.
 \end{lemma}
 \begin{proof}
  Let $I= ( b_{1}, b_{2}, ..., b_{r})A\in \ker \phi $, where $b_{i}\in B$. Then $I A[X, X^{-1}]= A[X, X^{-1}]$. This implies that $b_{i}\in A[X, X^{-1}]\cap B= A $, for all $i$. So $I \subseteq A$. Similarly $I^{-1}\subseteq A$. Hence $I= A$.
 \end{proof}
 
 \begin{lemma}\label{2.3}
  The sequence (1.0)[respectively (1.1)] is exact, except possibly at the place $M\Pic A$ [respectively $K\Pic A$]. It is exact there too if the map $\Pic A\rightarrow \Pic B$ is surjective.
 \end{lemma}
\begin{proof}
 We have the following commutative diagram

\tiny
 $$\xymatrix @R=.35in @C=.2in
 { & 1 \ar[d] & 1 \ar[d] & 1 \ar[d] & 0 \ar[d] & 0 \ar[d] \\
 1 \ar[r] & U(A) \ar[r] \ar[d] & U(B) \ar[r] \ar[d] & \mathcal I(A, B) \ar[r] \ar[d] & \Pic A \ar[r] \ar[d] & \Pic B \ar[d] \\
 1 \ar[r] & U(A[X,X^{-1}]) \ar[r] \ar[d] & U(B[X,X^{-1}]) \ar[r] \ar[d] &  \mathcal I(A[X,X^{-1}], B[X,X^{-1}]) \ar[r] \ar[d] & \Pic A [X,X^{-1}] \ar[r] \ar[d] & \Pic B [X,X^{-1}] \ar[d] \\
 1 \ar[r]  & MU(A) \ar[r] \ar[d] & MU(B) \ar[r] \ar[d] &  M\mathcal I(A, B) \ar[r] \ar[d] & M\Pic A \ar[r] \ar[d] & M\Pic B \ar[d] \\
  & 1 & 1 & 1 & 0 & 0
 }
 $$
 \normalsize
 where the first two rows are exact and each column is exact. Now the result follows by chasing this diagram.
\end{proof}

\begin{lemma}\label{2.4}
 Let $A\subseteq B$ be a ring extension. The map $\Pic A\rightarrow \Pic B $ is surjective if any one of the following  conditions holds:
 \begin{enumerate}[(1)]
 \item $A \subseteq B$ is subintegral.
 \item $A\subseteq B$ is an integral, birational extension of domains with $\dim A\leq 1$.
\end{enumerate}
 \end{lemma}

\begin{proof}
 (1) See Proposition 7 of \cite{is}.
 
 (2) Let $K$ be the quotient field of $A$ and $B$. We have the commutative diagram
 
 $$\begin{CD}
    \mathcal I(A, K)  @>>> \Pic A     @>>> 0    \\
    @V\varphi(A, B, K )VV   @V\rho VV    \\ 
    \mathcal I(B, K)  @>>> \Pic B     @>>> 0        \\
   \end{CD}$$
   
   where $\varphi(A, B, K )$ is surjective by Proposition 2.3 of \cite{ss}. Hence $\rho$ is surjective.
\end{proof}

\begin{lemma}(cf. \cite{Oy}, Lemma 1.4.)\label{2.6}
Let $A\subseteq  B$ be a ring extension with $B$ reduced  and  $A$  quasinormal in $B$. Then $A$ is subintegrally closed in $B$.
 
\end{lemma}

\begin{proof}
  We have not assumed $B$ to be a domain. By Lemma \ref{2.1}, it is enough to show that $K\mathcal I(A,B)=0 $. We have the sequence 
  $$
1\rightarrow  KU(A) \rightarrow  KU(B) \rightarrow  K \mathcal  I(A,B)\stackrel{\alpha} \rightarrow  K\Pic A \stackrel{\beta}\rightarrow  K\Pic B 
$$
which is exact except possibly at the place $K\Pic A$. Since $A$ and $B$ are reduced, $KU(A)=0$ and $KU(B) =0$. In the proof of Lemma 1.4 of \cite{Oy}, it is shown that the map  $K\Pic A \rightarrow  K\Pic B $ is injective, i.e., $\ker \beta=0$. We have $\im \alpha\subseteq \ker \beta$.  Hence $K\mathcal I(A, B)= 0$.
\end{proof}
\begin{remark}\rm{
 In the above lemma we cannot drop the condition that $B$ is reduced. For example, consider the extension $A=K\subsetneq B=K[b]$ with $b^{2}=0$, where $K$ is any field. Since $M\Pic K=0$,  clearly $A$ is quasinormal in $B$. But $A$ is not subintegrally closed in $B$, because   $b^{2}=b^{3}=0\in K$, $b\notin K$.}
\end{remark}

\begin{lemma}\label{2.5}
 Let $A\subseteq B$ be a ring extension with $B$ a domain. Then the following statements hold:
 \begin{enumerate}[(1)]
  \item If $A$ is quasinormal in $B$ then $M\mathcal I(A, B)= 0$. 
  \item Suppose the extension $A\subseteq B$ is integral and birational with $\dim A\leq 1$, and $M\mathcal I(A, B)= 0$. Then $A$ is quasinormal in $B$.
 \end{enumerate}

 \end{lemma}

\begin{proof}
 (1) Since $A$ and $B$ are domains, $MU(A)= MU(B) \cong \mathbb Z$. By (1.0), $\im \eta \subseteq \ker \varphi $. As $A$ is quasinormal in $B$, $\ker \varphi =0$. This implies that $\im \eta =0$. We get $M\mathcal I(A, B)= 0$.
 
 (2) By Lemma \ref{2.4}(2) and Lemma \ref{2.3}, the sequence (1.0) is exact at $M\Pic A$ also. Since $M\mathcal I(A, B)= 0$, we get the result.
\end{proof}

\begin{theorem}\label{2.11}
  Let $A\subseteq B$ be an integral, birational extension of domains with $\dim A\leq 1$. Then $M\mathcal I(A, B)=0$ if and only if $A$ is subintegrally closed in $B$ and $A\subseteq B$ is anodal.
\end{theorem}

\begin{proof}
 If $\dim A=0$ then $A = B$ and the assertion holds trivially in this case. If $\dim A= 1$ then by Theorem 1.13 of \cite{Oy}, $A$ is quasinormal in $B$ if and only if  $A$ is subintegrally closed in $B$ and $A\subseteq B$ is anodal. We also have $A$ is quasinormal in $B$ if and only if  $M\mathcal I(A, B)=0$ by Lemma \ref{2.5}. Combining these two results we get the assertion.
\end{proof}

Next, in Theorem \ref{2.10} and Remark \ref{example}, we show that in general, the conditions $A$ is subintegrally closed in $B$ and $A\subseteq B$ is anodal are necessary but not sufficient.

\begin{lemma}\label{2.8}
\begin{enumerate}[(1)]
 \item The diagram 
 $$\xymatrix{\mathcal I(A, B) \ar[r]^{\psi} \ar[rd]^{\phi} & \mathcal I(A[X], B[X])\ar[d]^{\theta}\\ & \mathcal I(A[X, X^{-1}], B[X, X^{-1}]) }$$
 is commutative.
 \item The maps $\phi$, $\psi$ and $\theta$ are injective.
 \item $\phi$ is an isomorphism if and only if $\psi$ and $\theta$ are isomorphisms.
 \item If $\phi $ is an isomorphism, i.e., $M \mathcal I(A, B)= 0$, then $A$ is subintegrally closed in $B$.
 
\end{enumerate}

\end{lemma}
\begin{proof} (1) Since the maps are natural,  the diagram is commutative.\hspace{1cm}

 (2) $\phi$ is injective by Lemma \ref{2.2}. The injectivity of $\psi$ and $\theta$  follows by a similar argument as in Lemma \ref{2.2}. \hspace{1cm}
 
 (3) If $\psi$ and $\theta$ are  isomorphisms then clearly $\phi $ is an isomorphism. Conversely, suppose $\phi $  is an isomorphism. Then by simple diagram chasing we get that $\psi$ and $\theta$ are isomorphisms.\
 
 (4) If $\phi $ is an isomorphism then $\psi$ is an isomorphism. Hence by Lemma \ref{2.1}, $A$ is subintegrally closed in $B$.
\end{proof}

\begin{lemma}\label{2.9}
 Let $\mathfrak{a}$ be a $B$-ideal contained in $A$. Then the homomorphism $M\mathcal I(A, B)\rightarrow M\mathcal I(A/\mathfrak{a}, B/\mathfrak{a})$ is an isomorphism. 
\end{lemma}
\begin{proof}
 Clearly, $\mathfrak {a}[X, X^{-1}]$ is a $B[X, X^{-1}]$-ideal contained in $A[X, X^{-1}]$. We have \\$\mathcal I(A, B) \cong \mathcal I(A/\mathfrak{a}, B/\mathfrak{a})$ and $\mathcal I(A[X, X^{-1}], B[X, X^{-1}]) \cong \mathcal I(A/\mathfrak{a}[X, X^{-1}], B/\mathfrak{a}[X, X^{-1}])$ by Proposition 2.6  of \cite{rs}. Now by chasing a suitable diagram we get the result.
\end{proof}

\begin{theorem}\label{2.10}
 Let $A\subseteq B$ be an integral, birational extension of domains. Suppose $M\mathcal I(A, B)=0$. Then  $A$ is subintegrally closed in $B$ and $A\subseteq B$ is anodal.
\end{theorem}

\begin{proof}
  By Lemma \ref{2.8}(4), $A$ is subintegrally closed in $B$. To prove $A\subseteq B$ is anodal,  by Lemma 1.10 of \cite{Oy}, it is enough to show that for every intermediate ring $C$ between $A$ and $B$ such that $C$ is a finite $A$-module, the map $M\Pic A\rightarrow M\Pic (A/\cc) \times M\Pic C$ is injective, where $\cc$ is the conductor  of $C$ in $A$. We first claim that the map $\tau : M\mathcal I(A, C)\rightarrow M\mathcal I(A, B)$ is injective, where $C$ is any intermediate ring between $A$ and $B$.\
  
  We have the commutative diagram
   $$\begin{CD}
 1    @>>> \mathcal I(A,C )    @>>> \mathcal I(A[X, X^{-1}], C[X, X^{-1}])     @>>> M\mathcal I(A, C) @>>> 1    \\
 @.                         @VVV                   @VVV                 @VV\tau V                      @.\\
 1    @>>>   \mathcal I(A, B)    @>\phi>> \mathcal I(A[X, X^{-1}], B[X, X^{-1}])     @>>> M\mathcal I(A, B) @>>>1     
 \end{CD}$$
 where the first two vertical arrows are natural inclusions (because any invertible $A$-submodule of $C$ is also an  invertible $A$-submodule of $B$).
  
  Let $\bar{J}\in \ker \tau$, where $J\in \mathcal I(A[X, X^{-1}], C[X, X^{-1}])$. Then $J\in \im \phi$ and there exists $J_{1}\in \mathcal I(A, B)$ such that $J_{1}A[X, X^{-1}]= J$. Let $J_{1}= (b_{1}, b_{2},..., b_{r})A$ and $J= (f_{1}, f_{2},..., f_{s})A[X, X^{-1}]$ where $b_{i}\in B$ and $f_{i}\in C[X, X^{-1}]$. Then clearly $b_{i}\in B\cap C[X, X^{-1}]= C$ for all $i$. So $J_{1}\subseteq C$. Also $J_{1}^{-1}\subseteq C$. This implies that $J_{1}\in \mathcal I(A, C)$.  So $\bar{J} =0$. This proves the claim.\
   
   Since $M\mathcal I(A, B) =0$, $M\mathcal I(A, C)=0$. By Lemma \ref{2.9},  $M\mathcal I(A/\cc, C/\cc)=0$, where $\cc$ is the  conductor of $C$ in $A$. By (1.0), we have $MU(A)\cong MU(C)$ and $MU(A/\cc)\cong MU(C/\cc)$. Now the result follows from the following exact sequence which we obtain by applying M to the unit-Pic sequence (\cite{wei 1}, Theorem 3.10),
 $$MU(A)\rightarrow MU(A/\cc) \times MU(C)\rightarrow MU(C/\cc)\rightarrow M\Pic A\rightarrow M\Pic (A/\cc)\times M\Pic C$$
 \end{proof}

\begin{remark}\label{example}\rm{
The converse of the above theorem holds for $\dim A\leq 1$ as seen in Theorem \ref{2.11}. In general, the converse does not hold. This is seen by considering Example 3.5 of C. Weibel \cite{wei}. In that example $A$ is a 2-dimensional noetherian domain whose integral closure is $B= K[X, Y]$, where $K$ is a field. So $A\subseteq B$ is an integral, birational extension. By Proposition 3.5.2 of \cite{wei}, $A\subseteq B$ is anodal and $A$ is subintegrally closed in $B$. Since $B$ is a UFD, $\Pic B = \Pic B[T, T^{-1}]= 0$. Then we get the exact sequence

$$
 1\rightarrow  MU(A) \rightarrow  MU(B) \rightarrow  M \mathcal  I(A,B) \rightarrow  M \Pic A \rightarrow  0.
$$
As $A$, $B$ are domains, $MU(A)= MU(B)\cong \mathbb Z$. So $M \mathcal I(A, B)\cong M\Pic A$. By Remark \ref{we1},  $L\Pic A\cong M\Pic A$.  Hence by Proposition 3.5.2  of \cite{wei},  $M\mathcal I(A, B)\neq 0$.}
\end{remark}

\section{Some observations on $M\mathcal I(A, B)$}
In this section we discuss some properties of the cokernel $M\mathcal I(A, B)$ in the general case.

Recall from [6, Section 3] that for any extensions $A\subseteq C \subseteq B$  of rings, we have the exact sequence 
$$ 1\rightarrow \mathcal I(A, C) \rightarrow \mathcal I(A, B) \stackrel{\varphi(A, C, B)}\rightarrow \mathcal I(C, B) $$
where the map $\varphi(A, C, B)$ is given by $\varphi(A, C, B)(I)= IC$.

Now it is natural to ask  under what conditions on $A\subseteq B$ the map $\varphi(A, C, B)$ is surjective. In \cite{bs}, Singh has proved that if $B$ is subintegral over $A$ then the map $\varphi(A, C, B)$ is surjective. In the next Proposition  we generalize Singh's result as follows: 

\begin{proposition}\label{3.5}
 For all rings $C$ between $A$ and $B$ such that $C$ is subintegral over $A$, the map $\varphi(A, C, B)$ is surjective.
\end{proposition}

\begin{proof}
 We have the commutative diagram
 $$\begin{CD}
  1  @>>> U(A)  @>>> U(B)  @>>> \mathcal I(A, B) @>>> \Pic A   @>>> \Pic B  \\
  @.      @VVV        @VV=V         @VV\varphi(A,C, B) V             @VV\rho V   @VV=V  @. \\
  1  @>>> U(C)  @>>> U(B)  @>>>  \mathcal I(C, B) @>>> \Pic C  @>>> \Pic B  \\       
 \end{CD}$$
 
 Since $\rho$ is surjective by Lemma \ref{2.4}(1), the result follows by chasing the diagram. 
\end{proof}

The following result gives another case where the map  $\varphi(A, C, B)$ is surjective.

Recall that a local ring $A$ is {\bf Hensel} if every finite $A$-algebra $B$ is a direct product of local rings.

\begin{proposition}\label{3.6}
 Let $A\subseteq B$ be an integral extension  with $A$ Hensel local. Then for all rings $C$ with $A\subseteq C\subseteq B$ the map $\varphi(A, C, B)$ is surjective.
\end{proposition}

\begin{proof}By Lemma 2.2 of \cite{ss}, it is enough to show that $ \varphi (A,D,B)$ is surjective for every subring $D$ of  $C$ containing $A$ such that $D$ is finitely generated as an $A$-algebra. Let such a ring $D$ be given. Since $D$ is integral over $A$,  $D$ is a finite $A$-algebra. As $A$ is Hensel, $D$ is a finite  direct product of local rings. Then $\Pic A$ and $\Pic D$ are both trivial. This implies that $\mathcal I(A, B)= U(B)/U(A)$,  $\mathcal I(D, B)= U(B)/U(D)$ and clearly $\varphi(A, D, B)$ is surjective.
\end{proof}

\begin{proposition}\label{3.7}
 Let $A\subseteq C \subseteq B $ be  extensions of rings with $A\subseteq C$ subintegral. Then the sequence 
 $$ 1\rightarrow M\mathcal I(A,C) \rightarrow M\mathcal I(A, B) \rightarrow M\mathcal I(C, B) \rightarrow 1$$
 is exact. 
\end{proposition}

\begin{proof}
 Consider the commutative diagram
 
\tiny
$$\xymatrix @R=.35in @C=.25in{ & 1\ar[d] & 1\ar[d] & 1\ar[d]  \\
 1 \ar[r] & \mathcal I(A, C)   \ar[r] \ar[d] & \mathcal I(A[X, X^{-1}],C[X, X^{-1}]) \ar[r] \ar[d] & M\mathcal I(A, C) \ar[r] \ar[d] & 1 \\
 1 \ar[r] & \mathcal I(A, B)   \ar[r] \ar[d] & \mathcal I(A[X, X^{-1}],B[X, X^{-1}]) \ar[r] \ar[d] & M\mathcal I(A, B) \ar[r] \ar[d] & 1 \\
 1 \ar[r] & \mathcal I(C, B)   \ar[r] \ar[d] & \mathcal I(C[X, X^{-1}],B[X, X^{-1}]) \ar[r] \ar[d] & M\mathcal I(C, B) \ar[r] \ar[d] & 1 \\
 & 1 & 1 & 1 
  }$$
  \normalsize
where the rows are clearly exact. Since $A\subseteq C$ is subintegral, so is $A[X, X^{-1}]\subseteq C[X, X^{-1}]$. Therefore by Proposition \ref{3.5}, the first two columns are exact.  Hence exactness of the last column follows by chasing the diagram.
 \end{proof}
 
 \begin{corollary}\label{Lnew}
  Let $A\subseteq B$ be a ring extension and let $^{^+}\!\!\!A$  denote the subintegral closure of $A$ in $B$. Then the sequence $$ 1\rightarrow M\mathcal I(A, ^{^+}\!\!\!A) \rightarrow M\mathcal I(A, B) \rightarrow M\mathcal I(^{^+}\!\!\!A, B)\rightarrow 1$$ is exact.
 \end{corollary}
 
 \begin{proof}
  Immediate from Proposition \ref{3.7}.
\end{proof}

\begin{proposition}\label{4.8}
  Let $A\subseteq B$ be a ring extension. Assume that $A$ is subintegrally closed in $B$.   Then
  
  (1) $M\mathcal I(A, B) \cong  M\mathcal I(A[T], B[T])$.
  
  (2) $M\mathcal I(A, B)$ is a torsion-free abelian group if $B$ is a seminormal ring. 
  
  (3) $M\mathcal I(A, B)$ is a free abelian group if $B$ is a seminormal ring and $A$ is Hensel local.
  
  (4) $M\mathcal I(A, B)=0$ if $B$ is a seminormal domain and $A$ is Hensel local.
 \end{proposition}
 
 \begin{proof}
 
 (1) Since $A$ is subintegrally closed in $B$,  $A[X]$  is subintegrally closed in $B[X]$ by Corollary 1.6 of \cite{ss}. Therefore $A[X, X^{-1}]$ is subintegrally closed in $B[X, X^{-1}]$. We have the commutative diagram
 \footnotesize
 $$\begin{CD}
 1    @>>> \mathcal I(A,B )    @>>> \mathcal I(A[X,X^{-1}], B[X,X^{-1}])     @>>> M\mathcal I(A, B) @>>> 1    \\
 @.                         @VV\psi V                   @VV\xi V                 @VVV                      @.\\
 1    @>>>   \mathcal I(A[T], B[T])    @>>> \mathcal I(A[T][X, X^{-1}], B[T][X,X^{-1}])     @>>> M\mathcal I(A[T], B[T]) @>>> 1      
 \end{CD}$$
 \normalsize
  
 where $\psi$ and $\xi$ are isomorphisms by Lemma \ref{2.1}. Hence we get the result.

 (2)  As $A$ is subintegrally closed in $B$ and $B$ is a seminormal ring,  $A$ is seminormal. Then by Remark \ref{we1}, $L\Pic A \cong  M\Pic A $. Since any seminormal ring is reduced, $MU(A)= H^{0}(A)  $ and  $MU(B)= H^{0}(B)$ by Remark \ref{LU1}. Now, from (1.0), we have the exact sequence 
   $$ 
  1\rightarrow  H^{0}(A) \rightarrow H^{0}(B) \rightarrow  M\mathcal  I(A, B) \rightarrow  M\Pic A  $$
  
  \noindent where $ M\Pic A$ is a torsion-free abelian group by Corollary 2.3.1 of \cite{wei}. Let $T$ be the cokernel of the map $H^{0}(A)\rightarrow H^{0}(B)$. Then $$ 1 \rightarrow T \rightarrow  M\mathcal  I(A, B) \rightarrow  M\Pic A  $$ is exact and $T$ is a free abelian group by Proposition 1.3 of \cite{wei}. Therefore $M\mathcal  I(A, B)$ is a torsion-free abelian group.  \
  
  (3)By  Theorem 2.5 of \cite{wei}, $L\Pic A= 0$.  Since $A $ is seminormal, $M\Pic A = 0$. Then we have the exact sequence 
  $$1\rightarrow  H^{0}(A) \rightarrow H^{0}(B) \rightarrow  M\mathcal  I(A, B) \rightarrow  1  $$\\
  and $M\mathcal I(A, B)= \rm{Coker}[H^{0}(A) \rightarrow H^{0}(B)]$  is a free abelian group by Proposition 1.3 of \cite{wei}. \
  
  (4)  Since $B$ is a domain,  $H^{0}(A)= H^{0}(B)\cong \mathbb{Z}$. So $M\mathcal I(A, B)=0$.
 \end{proof}

  \begin{lemma}\label{LP}
   Let $A\subseteq B$ be a subintegral extension. Then the map $L\Pic A \rightarrow L\Pic B $ is surjective.
  \end{lemma}
  
  \begin{proof}
   Since $A\subseteq B$ is subintegral, so are $A[X]\subseteq B[X]$ and $A[X, X^{-1}]\subseteq B[X, X^{-1}]$. Then the maps $\Pic A[X] \times \Pic A[X^{-1}]\rightarrow \Pic B[X] \times \Pic B[X^{-1}]$ and $\Pic A[X, X^{-1}]\rightarrow \Pic B[X, X^{-1}]$ are surjective by Lemma \ref{2.4}(1). Hence we get the result by chasing the following commutative diagram 
   $$\xymatrix{\Pic A[X] \times \Pic A[X^{-1}] \ar[r] \ar[d] & \Pic A[X, X^{-1}] \ar[r] \ar[d] & L\Pic A \ar[r] \ar[d] & 1 \\
   \Pic B[X] \times \Pic B[X^{-1}] \ar[r] \ar[d] & \Pic B[X, X^{-1}] \ar[r] \ar[d] & L\Pic B \ar[r]  & 1 \\
     1 & 1   &  \\}$$
  \end{proof}

  \begin{theorem}\label{LP1}
   Let $A\subseteq B$ be a ring extension with $A$ Hensel local and $B$ seminormal. Then $M\mathcal I(A, B)\cong M\mathcal I(A,^{^+}\!\!\!A)\oplus M\mathcal I(^{^+}\!\!\!A, B)$, where $^{^+}\!\!\!A$ is the subintegral closure of $A$ in $B$. 
  \end{theorem}
  
  \begin{proof}
   By Lemma \ref{LP}, $L\Pic A \rightarrow L\Pic ^{^+}\!\!\!A $ is surjective. Since $A$ is Hensel local,  $L\Pic A= 0$ by Theorem 2.5 of \cite{wei}. Therefore $L\Pic ^{^+}\!\!\!A = 0$ and $M\Pic ^{^+}\!\!\!A =0$ because $^{^+}\!\!\!A$ is seminormal. Then by the same argument as Proposition \ref{4.8}(3), $ M\mathcal I(^{^+}\!\!\!A, B)$ is a free abelian group. Now the result follows from the following exact sequence (Corollary \ref{Lnew})
   $$ 1\rightarrow M\mathcal I(A, ^{^+}\!\!\!A) \rightarrow M\mathcal I(A, B) \rightarrow M\mathcal I(^{^+}\!\!\!A, B)\rightarrow 1$$
  \end{proof}

\textbf{Acknowledgement}: I would like to express my sincere thanks to Prof. Balwant Singh for the many fruitful discussions and his guidance. I would also like to thank the referee for making many comments which have improved the exposition. Further, I would like to thank CSIR, India for financial support.


\begin{thebibliography}{AAA}

  
 \bibitem{TA}  T. Asanuma,  $\Pic R[X, X^{-1}]$ for $R$ a one-dimensional reduced noetherian ring, J. Pure and Applied  Algebra 71 (1991) 111-128.
 \bibitem{is}  F. Ischebeck,  Subintegral ring extensions and some K-theoretical functors, J. Algebra 121 (1989) 323-338.

 \bibitem {Oy}  N. Onoda and K. Yoshida, Remarks on quasinormal rings, J. Pure and Applied  Algebra 33 (1984) 59-67. 

 \bibitem  {rs}  L. G. Roberts and B.  Singh, Subintegrality, invertible modules and the Picard group, Compositio Math. 85 (1993) 249-279.

\bibitem{ss}    V. Sadhu and B. Singh,  Subintegrality, invertible modules and  polynomial extensions, J. Algebra 393 (2013) 16-23.

\bibitem {bs} B.  Singh, The Picard group and subintegrality in positive characteristic, Compositio Math. 95 (1995) 309-321.

 \bibitem {swan} R. G Swan, On seminormality, J. Algebra 67 (1980) 210-229.
 
 \bibitem{mv}  M. A. Vitulli, Weak normality and seminormality. Commutative algebra-Noetherian and non-Noetherian perspectives, 441–480, Springer, New York, 2011.
 
\bibitem {wei} C. Weibel, Pic is a contracted functor, Invent math. 103 (1991) 351-377.
 \bibitem{wei 1} C. Weibel, The K-Book: An Introduction to Algebraic K-Theory. Graduate Studies in Math. vol. 145, AMS, 2013.

\end{thebibliography}
\end{document}